%% file: 2014arxivContracts.tex
\documentclass[letterpaper, 10 pt, conference]{ieeeconf}  

\IEEEoverridecommandlockouts                              
\usepackage{Lilypack}
\usepackage{eucal}
\newcommand{\Prb}{\mathrm{P}}
\newcommand{\ol}{\overline}
\newcommand{\ul}{\underline}
\newcommand{\thetaH}{\ol{\theta}}
\newcommand{\thetaL}{\ul{\theta}}

\newcommand{\Hc}{\mathcal{H}}
\newcommand{\Rbb}{\mathbb{R}}
\newcommand{\q}{q}

\title{\LARGE \bf
Privacy and Customer Segmentation in the Smart Grid
}

\author{Lillian J. Ratliff, Roy Dong, Henrik Ohlsson, Alvaro A. C\'{a}rdenas and S. Shankar Sastry
\thanks{L. J. Ratliff, R. Dong, H. Ohlsson, and S. S. Sastry are with Faculty of Electrical Engineering and Computer Sciences, 
        University of California, Berkeley, Berkeley, CA, 94707, USA
        {\tt\footnotesize $\{$ratliffl, roydong, ohlsson, sastry$\}$@eecs.berkeley.edu}}%
        \thanks{A. A. C\'{a}rdenas is with the Department of Computer Science, University of Texas, Dallas,
        Dallas, TX 75080, USA
        {\tt\footnotesize alvaro.cardenas@utdallas.edu}}%
}

\begin{document}

\maketitle
\thispagestyle{empty}
\pagestyle{empty}

\begin{abstract}

In the electricity grid, networked sensors which record and transmit increasingly high-granularity data are being deployed. In such a setting, privacy concerns are a natural consideration. 
We present an attack model for privacy breaches, and, using results from estimation theory, derive theoretical results ensuring that an adversary will fail to infer private information with a certain probability, independent of the algorithm used. 
We show utility companies would benefit from less noisy, higher frequency data, as it would improve various smart grid operations such as load prediction. We provide a method to quantify how smart grid operations improve as a function of higher frequency data.  In order to obtain the consumer's valuation of privacy, we design a screening mechanism consisting of a menu of contracts to the energy consumer with varying guarantees of privacy. The screening process is a means to segment customers. Finally, we design insurance contracts using the probability of a privacy breach to be offered by third-party insurance companies. 

\end{abstract}

\section{Introduction}
Increasingly advanced metering infrastructure (AMI) is replacing older technology in the electricity grid. Smart meters send detailed information about consumer electricity usage over a network every half-hour, quarter-hour, or in some cases, every five minutes. This high-granularity data is needed to support energy efficiency efforts as well as demand-side management. 
However, improper handling of this 
information could also lead to unprecedented invasions of consumer privacy~\cite{quinn:2009aa, salehie:2012aa}.

Given that smart grid operations inherently have privacy and security risks~\cite{salehie:2012aa},
it would benefit the utility company, to know the answer to the following questions: How do consumers in the population value privacy? How can we quantify privacy? 
How do privacy-aware policies impact smart grid operations? In this paper we address these questions as well as expose new directions for future research on privacy and customer segmentation in the smart grid. 


Using our results on the fundamental limits of non-intrusive load monitoring~\cite{dong:2013aa}, we are able to come up with probabilities for the success of an attack by an adversarial agent independent of the algorithm. Then using these probabilities we can design a screening mechanism consisting of a menu of contracts to be offered to consumers.  One set of contracts to be offered by the utility company assess how the consumer values privacy thereby revealing his preferences. 
Based on their valuation of privacy as a good, consumers can select the quality of the service contract with the utility company. Essentially, electricity service is offered as a product line differentiated according to privacy where consumers can select the level of privacy that fits their needs and wallet. The screening process is a way to do customer segmentation the result of which can lead to targeting.

In particular, using knowledge of consumer preferences, the utility company could then incentivize consumers based on their preferences to choose a low privacy setting which helps increase the granularity of data for use by the utility company for programs like demand response, direct load control, etc. In addition, third-party insurance companies can design insurance contracts. Insurance allows the consumer to protect himself in the event of a privacy breach, i.e. she will be compensated for any experienced loss. 

The paper is organized as follows. In Section~\ref{sec:privacy} we review the problem of non-intrusive load monitoring (NILM) and show how NILM leads to our novel privacy metric. We show in an example that the probability of an adversary successfully implementing a privacy breach decreases with a decrease in sampling rate. We discuss the impact of sampling rate on smart grid operations in Section~\ref{sec:smart_grid}. In Section~\ref{sec:screen} we use the privacy metric to design a screening mechanism that consists of privacy contracts between the consumer and the utility company. Similarly, in Section~\ref{sec:insurance} we use the privacy metric to design insurance contracts. Finally, in Section~\ref{sec:conclusion} we summarize the results and discuss future research directions.


\input{privacy}
\input{smart_grid}


\section{Privacy Contracts}
\label{sec:screen}

In this section, we discuss how the utility company can design a screening mechanism in order to assess the consumer's unknown type. This is a mechanism design problem with asymmetric information. The utility company designs a screening mechanism whose contracting device is the privacy setting offered to the consumer. This screening process can be thought of as customer segmentation since it will extract each consumer's type after which the consumers can be grouped according to their type.



We consider a model in which there are only two types and we utilize standard results from the theory of screening (see, e.g.,~\cite{weber:2011aa}) to develop a framework for designing privacy contracts. We remark that as a result of the screening process the utility company will know how each consumer values privacy and can leverage that in the design of incentives aimed at inducing the consumer to select a privacy setting more desirable from the perspective of the utility company.

\subsection{Two Types: High-Privacy and Low-Privacy Settings}
\label{subsec:contract}
We model privacy-settings on smart meters as a good. The \emph{quality} of the good is either a high-privacy setting $\q_H$ or a low-privacy setting $\q_L$. The consumer can choose either a high privacy setting or a low privacy setting, i.e. the consumer selects $\q\in\mc{Q}=\{\q_H, \q_L\}\subset \mb{R}$ where $\q_L\leq \q_H$ and $-\infty<\q_L<\q_H<\infty$.
  The consumer's valuation of privacy is his \emph{type} which takes values $\theta\in\{\thetaL, \thetaH\} \subset \mb{R}$ where $\theta$ represents how much the consumer values high-privacy over low-privacy and $\thetaL<\thetaH$. 
 We assume that type $\theta$ is distinct from the private information itself; by this we mean that how much the consumer values privacy is not also private information. We note here that these types implicitly make use of the probabilities presented in Section \ref{sec:privacy}.

The consumers type $\theta$ is related to his willingness to pay in the following way: if the utility company announces a price $t$ for choosing $\q$, the type-dependent consumer's utility is equal to zero if he does not select a privacy setting $\q$, and it is 
\begin{equation}
  U(\q,\theta)-t\geq 0
  \label{eq:ir}
\end{equation}
if he does select a privacy setting. The case in which the consumer does not select a privacy setting is considered the \emph{opt-out} case in which consumer exercises his right to not participate. 
The inequality in \eqref{eq:ir} is often called the \emph{individual rationality} constraint and in the design of the privacy contracts we will enforce it in order to make sure that all consumers will opt-in. The function $U:\mb{R}\times \Theta\rar \mb{R}$ is assumed to be strictly increasing in $(\q, \theta)$, concave in $\q$, and represents the consumer's preferences. 

Since we have only two types, the contracts offered will be indexed by the privacy settings $\q_L$ and $\q_H$.
Further, as we mentioned the consumer can opt-out by not selecting a privacy option at all. Hence, we need to constrain the mechanism design problem by enforcing the inequality given in Equation~\eqref{eq:ir} for each value of $\theta\in \{\thetaL, \thetaH\}$. In addition, we need to enforce \emph{incentive-compatibility} constraints  
\begin{equation}
  U(\q_H, \thetaH)-t_H\geq U(\q_L, \thetaH)-t_L
  \label{eq:ic1}
\end{equation}
and 
\begin{equation}
  U(\q_L, \thetaL)-t_L\geq U(\q_H, \thetaL)-t_H
  \label{eq:ic2}
\end{equation}
where the first inequality says that given the price $t_H$ a consumer of type $\thetaH$ should prefer the high-privacy setting $\q_H$ and the second inequality says that given the price $t_L$ a consumer of type $\thetaL$ should prefer the low-privacy setting $\q_L$.

The utility company has unit utility
\begin{equation}
  v(\q, t)=-g(\q)+t
  \label{eq:uc_utility}
\end{equation}
where we assume that the function $g:\mc{Q}\rar\mb{R}$ is the unit cost to the utility company for the privacy setting $\q$. We assume that it is a strictly increasing, continuous function which is reasonable because, as we have mentioned in Section \ref{sec:smart_grid}, a low-privacy setting $\q_L$ provides the utility company with the high-granularity data it needs to efficiently operate and maintain the smart grid. For example, recall Section~\ref{subsec:dlc} in which we show that the performance of DLC degrades with decreases in sampling rate.

The screening problem is to design the contracts, i.e. $\{(t_L, \q_L), (t_H, \q_H)\}$ where $t_L, t_H\in \mb{R}$, so that the utility companies expected profit is maximized where the expected profit is
\begin{equation}
  \Pi(t_L, \q_L, t_H, \q_H)=(1-p)v(\q_L,t_L)+pv(\q_H, t_H)
  \label{eq:profit}
\end{equation}
where $p=\Prb(\theta=\thetaH)=1-\Prb(\theta=\thetaL)\in (0, 1)$ where $\Prb(\cdot)$ denotes probability.

To find the optimal pair of contracts, we solve the following optimization problem:
\begin{align} 
  \max_{\{(t_L, \q_L), (t_H, \q_H)\}}& \Pi(t_L, \q_L, t_H, \q_H)\tag{P-1} \label{eq:P1}
\\
\text{s.t.}\qquad&\ U(\q_H, \thetaH)-t_H\geq U(\q_L, \thetaH)-t_L\tag{IC-1}\label{eq:IC1}\\
&\ U(\q_L, \thetaL)-t_L\geq U(\q_H, \thetaL)-t_H\tag{IC-2}\label{eq:IC2}\\
&\ U(\q_L,\thetaL)-t_L\geq 0\tag{IR-1}\label{eq:IR1}\\
&\ U(\q_H,\thetaH)-t_H\geq 0\tag{IR-2}\label{eq:IR2}\\
&\ \q_L\leq \q_H\notag
 \end{align}

 Depending on the form of $U(\q, \theta)$ and $g(\q)$ problem~\eqref{eq:P1} can be difficult to solve.  
 So, we examine the constraints and try to eliminate as many as we can. 
 
First, we show that \eqref{eq:IR1} is active. Indeed, suppose not. Then, $U(\q_L, \thetaL)-t_L>0$ so that, from the first incentive compatibility constraint \eqref{eq:IC1}, we have
\begin{equation}
  U(\q_H, \thetaH)-t_H\geq U(\q_L, \thetaH)-t_L\geq  U(\q_L, \thetaL)-t_L>0
  \label{eq:binding}
\end{equation}
where the second to last inequality holds since $U(\q, \theta)$ is increasing in $\theta$ by assumption. 
As a consequence, the utility company could increase the price for both types since neither incentive compatibility constraint would be active. This would lead to an increase in the utility company's pay-off, i.e. a contradiction. Now, since $U(\q_L, \thetaL)=t_L$, the last inequality in \eqref{eq:binding} is equal to zero. This implies that \eqref{eq:IR2} is redundant.
Further, this argument implies that the constraint \eqref{eq:IC1} is active. 
Indeed, again suppose not. Then, 
\begin{equation}
   U(\q_H, \thetaH)-t_H> U(\q_L, \thetaH)-t_L\geq  U(\q_L, \thetaL)-t_L=0
  \label{eq:binding2}
\end{equation}
so that it would be possible for the utility company to decrease the incentive $t_H$ without violating \eqref{eq:IR2}. 

Now, let us assume that the marginal gain from raising the value of the privacy setting $\q$ is greater for type $\thetaH$, i.e. $U(\q, \thetaH)-U(\q, \thetaL)$ is increasing in $\q$.  Then, since \eqref{eq:IC1} is active, we have
\begin{equation}
  t_H-t_L=U(\q_H, \thetaH)-U(\q_L, \thetaH)\geq U(\q_H, \thetaL)-U(\q_L, \thetaL).
  \label{eq:binding3}
\end{equation}
This inequality implies that we can ignore \eqref{eq:IC2}. Further, since $U$ is increasing in $(\q, \theta)$ and we have assumed that $\thetaH>\thetaL$, we can remove the constraint $\q_L\leq \q_H$. We have reduced the constraint set to
\begin{align}
  t_H-t_L&=U(\q_H, \thetaH)-U(\q_L, \thetaH)\label{eq:newconst1}\\
 t_L&= U(\q_L, \thetaL)
 \label{eq:newconst2}
\end{align}
Thus, the optimization problem becomes 
\begin{align}
  \label{eq:newopt}
  \max_{(\q_L, \q_H)}&\bigg\{p(U(\q_H, \thetaH)-g(\q_H)-U(\q_L, \thetaH)+U(\q_L, \thetaL))\tag{P-2}\\
  &\quad (1-p)(U(\q_L, \thetaL)-g(\q_L))\bigg\}\notag
\end{align}
This reduces further to two independent optimization problems:
\begin{align}
  \label{eq:indpopts}
  &\max_{q_H}\{U(\q_H, \thetaH)-g(\q_H)\}\tag{P-3a}\\
  &\max_{q_L}\{-p(U(\q_L, \thetaH)-U(\q_L, \thetaL))+ \tag{P-3b}\\
  &\qquad\qquad(1-p)(U(\q_L, \thetaL)-g(\q_L))\} \notag
  \end{align}
\subsection{Direct Load Control Example}
\label{subsec:dlcex}
Recall that the unit gain the utility company gets out of the privacy setting $\q$ is a function $g:\mc{Q}\rar\mb{R}$. In this section, we discuss a particular example in which $g$ is a metric for how access to high-granularity data affects direct load control. In Section~\ref{subsec:dlc}, Figure~\ref{fig:equi_fig_1} shows how that as you decrease the sampling rate (increase the sampling interval) the performance degrades, i.e. the $\mc{H}_\infty$ norm increases, and it degrades in a linear way.  Hence, this motivates a choice for $g$ such that $g(\q_L)>g(\q_H)$ and decreases in a linear way. Hence, for this example, let \begin{equation}
  g(\q)=\zeta \q
  \label{eq:gxdlc}
\end{equation}
where $0<\zeta<\infty$.
Note that a decreased sampling rate corresponds to a higher privacy setting. The function $g$ as defined is increasing in $\q$ so that $g(\q_L)>g(\q_H)$. 

Assume that the consumer's utility is 
given by
\begin{equation}
  U(\q, \theta)=\frac{1}{2}(\bar{q}^2-(\q-\bar{\q})^2)\theta
  \label{eq:agentutility}
\end{equation}
where $q\in [0, \bar{q}]$ so that it is proportional to how close they are to the maximum privacy setting $\bar{q}$, and their type $\theta$. Suppose that $\theta\in\{\thetaL, \thetaH\}$ where $0<\thetaL<\thetaH$. Note that at $\thetaH$ the utility is $\frac{1}{2}\bar{q}^2$ and at $\thetaL$ the utility is zero.
$U$ satisfies the assumption that it is increasing. Let $p=\text{P}(\theta=\thetaH)$. Then, the optimal solutions to the screening problem are
\begin{equation}
  (\q_H^\ast, \q_L^\ast)=\left( \bar{q}-\frac{\zeta}{\thetaH}, \left[\bar{q}+\frac{(1-p)\zeta}{(p\thetaH-\thetaL)}\right]_+ \right)
  \label{eq:optH}
\end{equation}
where $\q_L^\ast=0$ when the probability $p$ is greater than the critical point $p^\ast=\thetaH/\thetaL$. 
The optimal prices $t_H^\ast, t_L^\ast$ can be found by plugging $(\q_H^\ast, \q_L^\ast)$ into \eqref{eq:newconst1} and \eqref{eq:newconst2}. 
If the utility company knew the types, then the optimal solution would be
\begin{equation}
   (\q_H^\dagger, \q_L^\dagger)=\left( \bar{q}-\frac{\zeta}{\thetaH},\bar{q}-\frac{\zeta}{\thetaL} \right)
  \label{eq:infobest}
\end{equation}
\begin{figure}[h]
  \begin{center}
    \begin{tikzpicture}
\draw[->] (0,0) -- (5,0) node[anchor=north] {$p$};
\draw	(0,0) node[anchor=north] {0}
		(2.5,0) node[anchor=north] {$p_0$}
		(4.5,0) node[anchor=north] {1};
\draw (4.5,0.1) -- (4.5, -0.1);
\node at (4.75,1.85) {$\q_L^\dagger$};
\node at (2.25, 1.05) {$\q_L^\ast$};
\node at (5.25, 2.5) {$\q_H^\ast=\q_H^\dagger$};
\draw[thick] (0,2.5) -- (4.5, 2.5);
\draw[dotted, thick] (0,1.85) -- (4.5, 1.85);
\draw[->] (0,0) -- (0,3) node[anchor=east] {$\q$};
\node at (-1, 2.5) {$\bar{\q}-\frac{\zeta}{\thetaH}$};
\node at (-1, 1.85) {$\bar{\q}-\frac{\zeta}{\thetaL}$};
\draw[thick] (0,1.85) parabola (2.5,0);
\end{tikzpicture}  
\end{center}
  \caption{Comparison between full information and asymmetric information solutions as a function of $p$ the probability of the high-type in the population.}
  \label{fig:info}
\end{figure}
In Figure~\ref{fig:info}, we show that as the probability of the high-type being drawn from the population increases, $q_L^\ast$ decreases away from the optimal full information solution $q_L^\dagger$. This occurs until $p=p_0=\thetaL / \thetaH$ 
which is the critical probability. After $p_0$, $q_L^\ast=0$ and remains there until $p=1$.

The social welfare is defined to be sum of the pay-off to the utility company and to the consumer.
The social welfare is given by
\begin{align}
  &\Psi^\ast(p)=\Pi(t_L^\ast, q_L^\ast, t_H^\ast, q_H^\ast)+p(U(q_H^\ast, \thetaH)-t_H^\ast).
  \label{eq:sw}
\end{align}
Let 
\begin{align}
  &\vphi=(1-p)\Bigg( \frac{1}{2}\left(\bar{q}^2-\left( \frac{(1-p)\zeta}{p\thetaH-\thetaL} \right)^2 \right)\thetaL\notag\\
  &\qquad\quad-\zeta\left( \bar{q}-\frac{(1-p)\zeta}{p\thetaH-\thetaL} \right)\Bigg).
  \label{eq:vphi}
\end{align}
Then,
\begin{align}
  \Psi^\ast(p)&=-p\zeta\left( \bar{q}-\frac{\zeta}{\thetaH} \right)\notag\\
  &+\left\{\begin{array}{ll}
    \hspace*{-0.15cm}\vphi,&\hspace*{-0.3cm} p\leq p_0\\
   \hspace*{-0.15cm} p\left( \frac{1}{2}\left( \bar{q}^2-\left( \frac{(1-p)\zeta}{p\thetaH-\thetaL} \right)^2\right)(\thetaH-\thetaL) \right), &\hspace*{-0.2cm} p>p_0\end{array}\right.
  \label{eq:part1}
\end{align}
The social welfare reaches a critical point at $p_0$ beyond which the utility company will exclude the low-type from the market and only provide privacy contracts to the high-type. This is called the \emph{shutdown solution}. It is reasonable that as soon as the probability of the utility company facing a consumer of high-type reaches a critical point, they will focus all their efforts on this type of consumer since a high-type desires a higher privacy setting which results in a degradation of the DLC scheme. 
We remark that people who value high privacy more need to be compensated more to participate in the smart grid. If there are two contracts, then even consumers who do not value privacy much will have an incentive to lie. Through the screening mechanism, the consumer will report his type truthfully.

\section{Privacy Insurance Contracts}
\label{sec:insurance}
In this section, we will design an insurance contract, to be offered by a third-party company to the consumer, that uses the probability $\eta$ that an adversary will fail to infer private information about the consumer.
In the previous section we designed contracts to get consumers to allow for lower privacy settings; in this section we design insurance contracts that allow consumers to purchase protection against attacks given they know the probability of a successful attack occurring. The analysis that follows is well known in the economics literature (see, e.g., \cite{rothschile:1976aa, jaynes:1978aa, mussa:1978aa}). Using the theory of insurance contracts when there is asymmetric information and the probability that an adversary can gain access to consumers' private information, we analyze both the consumer's choice on how much insurance to invest in as well as the insurer's decision about which contracts to offer to a population with both high- and low-risk consumers.

\subsection{Analysis of Consumer's Decision}
Let us start by analyzing the decision the consumer would make about selecting an amount of insurance given knowledge of $\eta$. Let the consumer's utility function be denoted by $\tilde{U}:\mb{R}\rar \mb{R}$ and assume that $\tilde{U}$ is increasing, twice differentiable and strictly concave. Let us suppose that the consumer is \emph{risk-averse} which means that the consumer, who makes a decision under uncertainty, will try to minimize the impact of the uncertainty on her decision. 

In addition, suppose the consumer has initial wealth $y$, runs the risk of loss $\ell$ with probability $1-\eta$. In the context of our problem, wealth represents \emph{private~information} that can be gained through analysis of consumer energy consumption data and loss represents exposure of this private information. Recall that $1-\eta$ is the probability with which an adversarial agent could gain access to private information about a consumer through access to their consumption data. 

The consumer must decide how much insurance to buy. Let the cost of one unit of insurance be $c$ and suppose that the insurer pays the consumer $\beta$ in the event that an adversary attacks them resulting in an exposure of private information where $\beta$ is the amount of insurance the consumer agrees to buy. Then the consumer wants to solve the following optimization problem:
\begin{align}
  \max_{\beta\geq 0}\{\eta \tilde{U}(y-\beta c)+(1-\eta)\tilde{U}(y+(1- c)\beta-\ell)\}\tag{P-4}
  \label{eq:cinsp}
\end{align}
Suppose that $\beta^\ast$ is a local optimum, then there exists a Lagrange multiplier $\lambda$ such that  
\begin{align}
  \left\{\begin{array}{l}
    0=-\lambda-\eta c \tilde{U}'(y-\beta^\ast c)\\
    \qquad\quad+(1-\eta)(1-c)\tilde{U}'(y+(1-c)\beta^\ast-\ell)\\
  0=\lambda\beta^\ast\\
  0\geq \beta^\ast\\
  0\geq \lambda
\end{array}\right.
  \label{eq:KKT}
\end{align}
These conditions are the Karush-Khun-Tucker (KKT) necessary conditions. Combining the first and the last condition, we get
\begin{equation}
  0\geq -\eta c \tilde{U}'(y-\beta^\ast c)+(1-\eta)(1-c)\tilde{U}'(y+(1-c)\beta^\ast-\ell)
  \label{eq:KKT1}
\end{equation}
We analyze the consumer's decision by considering two cases and we present the results in the following propositions.
\begin{prop}
  \label{prop:fair}
  Suppose that the consumer is offered privacy insurance at the rate $c=1-\eta$, i.e. at a rate equal to the probability of a successful attack. Then the consumer will choose to purchase an amount of insurance equal to the loss, i.e. $\beta^\ast=\ell$.
\end{prop}
\begin{proof}
  Since $c=1-\eta$, \eqref{eq:KKT1} reduces to 
  \begin{equation}
    0\geq\eta (1-\eta)\big(\tilde{U}'(y+\eta\beta^\ast-\ell) - \tilde{U}'(y-\beta^\ast (1-\eta))\big) 
    \label{eq:prop1-1}
  \end{equation}
  and since $(1-\eta)\eta\geq 0$ this again reduces to
  \begin{equation}
    0\geq \tilde{U}'(y+\eta\beta^\ast-\ell) - \tilde{U}'(y-\beta^\ast (1-\eta))
    \label{eq:prop1-2}
  \end{equation}
  Recall that we assumed $\tilde{U}$ to a be a concave function and that a function is strictly concave if and only if its derivative $\tilde{U}'$ is decreasing. Hence,
  \begin{equation}
    \tilde{U}'(z)<\tilde{U}'(z-\ell)
  \label{eq:prop1-3}
\end{equation}
since $\ell>0$. This fact along with \eqref{eq:prop1-2} implies that $\beta^\ast>0$. Now, we claim that $\beta^\ast=\ell$. Indeed, suppose that $0<\beta^\ast<\ell$, then from \eqref{eq:prop1-2} we have
\begin{equation}
  \tilde{U}'(\tilde{y}-\ell)\leq \tilde{U}'(\tilde{y}-\beta^\ast)
  \label{eq:prop1-4}
\end{equation}
where $\tilde{y}=y+\eta\beta^\ast$. This inequality violates \eqref{eq:prop1-3}. On the other hand, suppose that $0\leq \ell\leq \beta^\ast$, then from \eqref{eq:prop1-3} we have
\begin{equation}
  \tilde{U}'(\tilde{y}-\beta^\ast)>\tilde{U}'(\tilde{y}-\ell)
  \label{eq:prop1-5}
\end{equation}
but this violates the KKT inequality \eqref{eq:prop1-2}. Hence, $\beta^\ast=\ell$ which is to say that the consumer will purchase an amount of insurance equal to the loss of privacy she would endure under an attack.
\end{proof}
\begin{prop}
  \label{eq:unfair}
  Suppose that the consumer is offered insurance at the rate $c>1-\eta$, i.e. at a rate higher than the probability of a successful attack. Then the consumer will not purchase the full insurance, i.e. $\beta^\ast<\ell$.
\end{prop}
\begin{proof}
  Suppose that the consumer is offered privacy insurance at a rate $c>1-\eta$ and that the optimal choice for the consumer is $\beta^\ast=\ell\geq 0$. Then, first-order optimality conditions imply that 
  \begin{equation}
    -\eta \tilde{U}'(y-\ell c) c+(1-\eta)\tilde{U}'(y-\ell c)(1-c)=0
    \label{eq:prop2-1}
  \end{equation}
  However, since $c>1-\eta$ and $\tilde{U}$ is increasing, from \eqref{eq:KKT1} we have
  \begin{equation}
    (-\eta c+(1-\eta)(1-c))\tilde{U}'(y-\ell c)<0
    \label{eq:prop2-2}
  \end{equation}
  so that, in fact, the optimal $\beta$ has to be less than the loss experienced, i.e. $\beta^\ast<\ell$.
\end{proof}

\subsection{Analysis of the Insurer's Decision}
Let us now consider the case of a third-party insurance company offering offering privacy insurance to the consumer which protects them against losses due to attacks. In a way, insurance allows the consumer to \emph{hedge their bet} against selecting a contract with a low-privacy setting. 

We consider a similar setup as before: the consumer's utility function $\tilde{U}$ is strictly concave, increasing and twice differentiable and for the sake of analysis we assume that $\tilde{U}(0)=0$. We consider a scenario in which the insurer faces two types: high-risk consumer $\theta_h$ and low-risk consumer $\theta_l$. That is to say we are assuming that there is a portion of the population that is more likely to be attacked, i.e. the risky consumers, possibly because they engage in high-risk behavior or due to the fact that they selected a low-privacy setting contract with the utility company. The consumer again has an initial amount of private information with value $y$ and with probability $1-\eta_j$ some of her private information is exposed resulting in a loss $\ell$ where $j=h,l$ indicates the consumer's type. We assume that $1-\eta_l<1-\eta_h$. 

We will assume that the insurer has a prior over the distribution of types. In particular, we assume that the risky type $\theta_h$ occurs in the population with probability $p$ and that $p>0$. 

Suppose we are given an insurance contract $(\alpha_a, \alpha_n)$ where $\alpha_a$ is the compensation to the consumer given that a successful attack occurred and $\alpha_n$ is the neutral case (no attack). Let $X$ be a random variable representing the consumer's wealth such that with probability $1-\eta_i$ it takes value $y-\ell+\alpha_a$ and with probability $\eta_i$ it takes value $y-\alpha_n$. Then, the consumers expected utility is
\begin{equation}
  E[\tilde{U}(X)]=(1-\eta_i)\tilde{U}(y-\ell+\alpha_a)+\eta_i\tilde{U}(y-\alpha_n)
  \label{eq:expectedU}
\end{equation}
Note that in the previous subsection we analyzed the consumer's decision given a insurance contract of the form 
\begin{equation}
\alpha_a=(1-c)\beta, \ \ \alpha_n=\beta c
  \label{eq:specContract}
\end{equation}
The insurer is a monopolist whose expected cost is 
\begin{align}
  \Pi(\alpha_a^h, \alpha_n^h,\alpha_a^l, \alpha_n^l)=& p\left(-(1-\eta_h)\alpha_a^h+\eta_h\alpha_n^h\right)\notag\\
  &\quad+(1-p)\left( -(1-\eta_l)\alpha_a^l+\eta_l\alpha_n^l \right)
  \label{eq:expectedMono}
\end{align}
In the case of asymmetric information, i.e. the insurer does not know the consumer's type, the optimization problem he must solve is 
\begin{align}
  \max_{\{(\alpha_a^j, \alpha_n^j)\}_{j=h,l}}&\Pi(\alpha_a^h, \alpha_n^h,\alpha_a^l, \alpha_n^l)\tag{P-5} \label{eq:optP-5}\\
  \text{s.t.}\qquad& (1-\eta_i)\tilde{U}(y-\ell+\alpha_a^i)+\eta_i \tilde{U}(y-\alpha_n^i)\notag\\
  &\quad \geq (1-\eta_i)\tilde{U}(y-\ell+\alpha_a^j)+\eta_i\tilde{U}(y-\alpha_n^j),\notag\\
  &\qquad\ i,j\in\{h,l\}, \ i\neq j\tag{IC}\label{eq:IC}\\
  &(1-\eta_i)\tilde{U}(y-\ell+\alpha_a^i)+\eta_i\tilde{U}(y-\alpha_n^i)\notag\\
  &\quad \geq (1-\eta_i)\tilde{U}(y-\ell)+\eta_i\tilde{U}(y), \ i\in \{h,l\}\tag{IR}\label{eq:IR}
 \end{align}
 Constraints labeled \eqref{eq:IC} are the incentive compatibility constraints and constraints \eqref{eq:IR} are the individual rationality constraints. Both are similar to those presented in Section \ref{subsec:contract}. Incentive compatibility ensures that the consumer will report their type truthfully and the individual rationality constraint ensures that the consumer will participate.

 Following a similar reasoning as in Section~\ref{subsec:contract}, we can reduce the optimization problem \eqref{eq:optP-5} by reasoning about the constraint set defined by \eqref{eq:IC} and \eqref{eq:IR}. In particular, we argued that the high-privacy type's incentive compatibility constraint was active and that the low-privacy type's individual rationality constraint was active. In addition, we showed the other two constraints \eqref{eq:IC2} and \eqref{eq:IR1} could be removed. Now, in the insurance case, since $1-\eta_l<1-\eta_h$, the incentive compatibility constraint for the risk type is active and the individual rationality constraint for the safe type is active, i.e.
 the constraint set for \eqref{eq:optP-5} becomes
 \begin{align}
   & (1-\eta_h)\tilde{U}(y-\ell+\alpha_a^h)+\eta_h \tilde{U}(y-\alpha_n^h)\notag\\
   &\quad = (1-\eta_h)\tilde{U}(y-\ell+\alpha_a^l)+\eta_h\tilde{U}(y-\alpha_n^l)\tag{IC-h}\\
   &(1-\eta_l)\tilde{U}(y-\ell+\alpha_a^l)+\eta_l\tilde{U}(y-\alpha_n^l)\notag\\
   &\quad = (1-\eta_l)\tilde{U}(y-\ell)+\eta_l\tilde{U}(y)\tag{IR-l}
   \label{eq:newconstraintP5}
 \end{align}
 Let us try to restate the problem in a way which allows us to characterize the solutions. Since we have assumed that $\tilde{U}$ is strictly concave, increasing and twice differentiable, we can define $W$ be its inverse, where $W'>0$ and $W''>0$. Further, define
\begin{equation}
  \tilde{U}_a^i=\tilde{U}(y-\ell+\alpha_a^i)\ \ \text{ and }\ \ \tilde{U}_n^i=\tilde{U}(y-\alpha_n^i).
  \label{eq:Uai}
\end{equation}
The transformed utility is
\begin{align}
  \widetilde{\Pi}(\tilde{U}_a^h, \tilde{U}_n^h, \tilde{U}_a^l, \tilde{U}_n^l)&=p\big(-\eta_hW(\tilde{U}_n^h)-(1-\eta_h)W(\tilde{U}_a^h)\notag\\
  &\ +x-(1-\eta_h)\ell\big) +(1-p)\big(-\eta_lW(\tilde{U}_n^l)\notag\\
  &\ \ -(1-\eta_l)W(\tilde{U}_a^l)+x-(1-\eta_l)\ell\big) 
  \label{eq:newcostIns}
\end{align}
Then problem \eqref{eq:optP-5} becomes
\begin{align}
  \label{eq:optP-6}
  \max_{\{(\tilde{U}_a^i, \tilde{U}_n^i)\}_{i=h,l}}&\widetilde{\Pi}(\tilde{U}_a^h, \tilde{U}_n^h, \tilde{U}_a^l, \tilde{U}_n^l)\tag{P-6}\\
  \text{s.t.}&\ (1-\eta_h)\tilde{U}_a^h+\eta_h \tilde{U}_n^h= (1-\eta_h)\tilde{U}_a^l+\eta_h\tilde{U}_n^l\notag\\
  &(1-\eta_l)\tilde{U}_a^l+\eta_l\tilde{U}_n^l= (1-\eta_l)\tilde{U}(y-\ell)+\eta_l\tilde{U}(y)\notag
\end{align}
The Lagrangian of the optimization problem is 
\begin{align}
  L&(\tilde{U}_a^h, \tilde{U}_n^h, \tilde{U}_a^l, \tilde{U}_n^l, \lambda_1, \lambda_2)=\widetilde{\Pi}(\tilde{U}_a^h, \tilde{U}_n^h, \tilde{U}_a^l, \tilde{U}_n^l)\notag\\
  &\quad+\lambda_1((1-\eta_h)\tilde{U}_a^h+\eta_h \tilde{U}_n^h-(1-\eta_h)\tilde{U}_a^l-\eta_h\tilde{U}_n^l)\notag\\
  &\quad+\lambda_2((1-\eta_l)\tilde{U}_a^l+\eta_l\tilde{U}_n^l-(1-\eta_l)\tilde{U}(y-\ell)).
  \label{eq:lag}
\end{align}
\begin{prop}
  Given the probabilities $1-\eta_j$, $j=h,l$ that the consumer of type $j$ will experience a privacy breach, if the insurer solves the optimization problem \eqref{eq:optP-6}, then the high-risk consumer will be fully insured and the low-risk consumer will not be fully insured.
\end{prop}
\begin{proof}
  We first show that the risky type will be fully insured. Taking the derivative of the Lagrangian with respect to $\tilde{U}_a^h$ and $\tilde{U}_n^h$ we get the following two equations:
\begin{align} 
  0&=-p(1-\eta_h)W'(\tilde{U}_a^h)+\lambda_1(1-\eta_h)\label{eq:uar}\\
  0&=-p\eta_h W'(\tilde{U}_n^h)+\lambda_1\eta_h
  \label{eq:2eq}
\end{align}  
Solving for $\lambda_1$ in the first equation and plugging it into the second, we get $\tilde{U}_a^h=\tilde{U}_n^h$ so that $\ell-\alpha_a^h=\alpha_n^h$, i.e. the amount the high-risk type pays for insurance is equal to the compensation minus the loss in the event of a privacy breach. Thus, the high-risk type will be fully insured. 

Now, we show that the low-risk type will not be fully insured. Taking the derivative of the Lagrangian with respect to $\tilde{U}_a^l$ and $\tilde{U}_n^l$, we get
\begin{align}
  0&=-(1-\eta_l)(1-p)W'(\tilde{U}_a^l)-\lambda_1(1-\eta_h)+\lambda_2(1-\eta_l)\label{eq:ual}\\
  0&=-(1-p)\eta_lW'(\tilde{U}_n^l)-\lambda_1\eta_h+\lambda_2\eta_l
  \label{eq:2eq2}
\end{align}
From \eqref{eq:uar}, we solved for $\lambda_1=pW'(\tilde{U}_a^h)$. By plugging in $\lambda_1$ into \eqref{eq:ual}, solving for $\lambda_2$ and plugging both $\lambda_1$ and $\lambda_2$ into \eqref{eq:2eq2}, we get the following expression:
\begin{align}
  0= &W'(\tilde{U}_a^h)p\left( -\eta_h+\eta_l\frac{1-\eta_h}{1-\eta_l} \right)\notag\\
  &\quad+\eta_l(1-p)(W'(\tilde{U}_a^l)-W'(\tilde{U}_n^l))
  \label{eq:eq3}
\end{align}
Since $\eta_l>\eta_h$ and $W'$ is increasing by assumption, the above equation implies that 
\begin{equation}
  \tilde{U}_n^l-\tilde{U}_a^l>0
  \label{eq:nofull}
\end{equation}
and hence the low-risk type does not fully insure.
\end{proof}
The above proposition tells us that in order to keep the high-risk type from masking as a low-risk type, the insurer must make the contract for the low-risk type unappealing to the high-risk type.

We remark that the analysis in this section can be applied to the case where the utility company is purchasing insurance as well. In particular, if the utlity company has not invested in a lot of security or tjeu are not following the best practices recommendations, e.g. NIST-IR 7628 \cite{group:2010aa}, then they are engaging in \emph{risky} behavior. The insurance company will not know a propri whether or not the utility company is high-risk type. Through the design of insurance contracts the insurance company can asses the utility's type while offering contracts that maximize their own utilty. 

\section{Conclusion}
\label{sec:conclusion}
Utilizing our results on the fundamental limits of non-intrusive load monitoring, we provide a novel upper bound on the probability of a successful privacy breach. 
Under the privacy metering policy in which sampling rate variation is used, we study how the performance of direct load control degrades using the $\mc{H}_\infty$ norm. This provided us with a metric for understanding how sampling rate affects the quality of direct load control.
Using this metric along with the upper bound on the probability for a successful privacy breach, we design a screening mechanism for the problem of obtaining the consumer's type when there is asymmetric information. Further, we design insurance contracts using the probability of successful privacy breach given that in the population of consumers there is both high-risk and low-risk consumers. 

This work opens up a number of questions in the area of privacy metrics as well as customer segmentation and targeting. 
We considered only two-type models in both the design of contracts. We are currently looking at the theory for a continuum of types.
The screening problem with a continuum of types results in a problem that resembles a partial differential equation constrained optimal control problem. We are developing numerical techniques to solve this problem. 
We also assumed that the utility company and private insurer knew the distribution of types in the population. We are currently developing algorithms for learning these probabilities using data-drive techniques. 
In addition, we considered that the utility would offer a contract solely based on privacy settings whereas in reality the contract would normally contain additional items such as maximum power consumption, rate, etc. Consumers in the population may value these goods differently. In this setting, the screening problem would be come multi-dimensional~\cite{basov:2005aa}. We are exploring this in the context of privacy-aware incentive design for behavior modification. 



\bibliographystyle{IEEEtran}
\bibliography{2014CDC}

\end{document}

%% file: privacy.tex
\section{Privacy Guarantees}
\label{sec:privacy}
In this section, we discuss our metric for privacy, and guarantees of privacy under this metric. For this paper, we restrict the scope of our analysis to data collection policies. Another important aspect of privacy is how data retention policies can alter privacy and smart grid performance. Such a topic is reserved for future research.

\subsection{Nonintrusive Load Monitoring}
Our formulation of privacy builds on recent research into nonintrusive load monitoring (NILM) algorithms, first proposed by Hart~\cite{hart:1992aa}. The goal of NILM is to use the aggregate power consumption signal, which can be measured by metering infrastructures without the placement of additional sensors inside the home, and make inferences on the load profile. For example, the problem of energy disaggregation is to recover the power consumption signals of individual devices~\cite{kolter:2011aa}. Another example would be to detect when devices switch on and off, which is often referred to as event-based NILM~\cite{anderson:2012aa}.

There are many approaches to the design of NILM algorithms, including but not limited to hidden Markov models (HMMs)~\cite{kolter:2012aa,johnson:2013aa,parson:2012aa}, sparse coding methods~\cite{kolter:2010aa}, and dynamical systems approaches~\cite{dong:2013ab}. These methods vary in the prior information required; some are completely unsupervised and nonparametric, while others require a large amount of disaggregated data to build a dictionary. However, in our recent work, we provided a unifying framework for modeling all of these algorithms~\cite{dong:2013aa}.

To the best of our knowledge, every approach to energy disaggregation gives devices a certain kind of model, be it an HMM, dictionary, or dynamical system. Most of these approaches have an input space: for HMMs, the input is a sequence of latent state transitions; for the sparse approaches, it is a sparse vector corresponding to the most representative elements in the dictionary, and in dynamical systems, it is an input signal to the systems.

In all these approaches, a fixed input yields a probability distribution on the device's power consumption; this distribution is often assumed to be Gaussian. Finally, note that in all the mentioned frameworks, recovering the input is almost the same as recovering the device's power consumption signal.

Other NILM formulations also can fall into this framework. For example, in event-based NILM, if the task is to determine whether or not the air conditioner is in use, we can consider the probability distribution across aggregate power consumption signals when the air conditioner is on and when the air conditioner is off. In this case, whether or not the air conditioner is on serves as an input.

In our framework, NILM algorithms are abstracted as measurable functions operating on the aggregate power consumption signal, and we ask whether or not the algorithms can successfully distinguish between inputs. For a more comprehensive treatment of this topic, we refer the reader to~\cite{dong:2013aa}.

\subsection{Privacy Metric}
\label{subsec:pm}
Now, we are ready to introduce our metric of privacy. 
%
%

One of the common theoretical definitions for a privacy metric is the notion of differential privacy~\cite{dwork:2006aa,le-ny:2014aa}.  While differential privacy has many attractive properties, it is most useful when we want to share data via a trusted third party aggregator, or by injecting noise in the original messages sent to a third party; however, for many practical, regulatory, dispute resolution, performance, or business reasons, there will always be several cases where we need to get access to the raw data, and in these cases differential privacy will not help us identify a good security mechanism to prevent raw data from being compromised. 



In contrast, we fix a definition of a privacy breach where the user has a set of possible inputs, and he wishes to keep the true input private. For example, the definition of privacy breach might be whether or not an adversary knows if the user is doing dishes in the dishwasher, watching TV, or exercising on a treadmill in the evening.  This notion of equivocation is related to recent work in privacy who measures privacy not with differential privacy but with equivocation metrics~\cite{shokri2011quantifying}.


We also assume a powerful adversary that knows the probability distribution of the energy consumption of digital signals, and their prior probability for being active in an aggregated signal.  As an observation, we assume the adversary has access to the aggregate power consumption signal, but not the device-level signals, for a building.

More formally, suppose there are two inputs $u_1$ and $u_2$. For each input, the aggregate power consumption signal follows distributions $F_1$ and $F_2$, respectively. In this paper, we assume these distributions to be Gaussian; however, we consider the general case in \cite{dong:2013aa}.

Suppose the distributions have means $\mu_1$ and $\mu_2$, and both distributions have the same covariance $\sigma^2 I$. 
Let $a = \sigma^{-2} \left( \mu_0 - \mu_1 \right)$ and $b = \frac{1}{2\sigma^2} \left( \|\mu_1\|_2^2 - \|\mu_2\|_2^2  \right)$. Suppose the adversary uses any estimator $\hat{u}$ and suppose the events $\{u=u_1\}$ and $\{u=u_2\}$ are equally likely. Then, the probability of our adversary successfully distinguishing two inputs is bounded by
 \begin{equation}
 \label{eq:norm_cdf}
 P(\hat{u}=u)\leq\frac{1}{2} \left( 1 - \operatorname{erf}\left(\frac{-\frac{1}{\|a\|_2}(a^T \mu_0 + b)}{\sqrt{2\sigma^2}}\right) \right)
 \end{equation}
 where $\operatorname{erf}$ is the Gauss error function. More details can be found in \cite{dong:2013aa}.

%% file: smart_grid.tex
\section{Smart grid operations}
\label{sec:smart_grid}
In Section \ref{sec:privacy}, we developed a metric for privacy. 
If privacy is the only thing of concern, a trivial solution is to record nothing transmit nothing. However, the utility company has other objectives than just preserving the privacy of its consumers. Hence, privacy issues arise because the sensitive data has other uses.
Such polices as noise injection or varying the sampling rate can be employed to protect against privacy breaches while still allowing the utility company to operate.

In advanced metering infrastructures, the data is used to improve the performance of smart grid operations. How smart grid operations degrade under different metering policies is an active topic of research; for preliminary investigations, see \cite{Cardenas2012,Dong2014}. Intuitively, the performance will degrade as fewer samples are collected or more noise is added. We attempt to quantify this degradation. In this section, we develop an direct load control example to demonstrate how smart grid operations performance is affected by different sampling rates.

\subsection{Direct Load Control}
\label{subsec:dlc}
The problem of direct load control has recently been studied as viable option to improve smart grid operations~\cite{Callaway2011,mathieu:2013aa}.

Generator output is generally determined by two processes: unit commitment and economic dispatch. \emph{Unit commitment} is done in advance, and sets the generator ramping schemes. \emph{Economic dispatch} is done online, and determines the output levels of generators that are already online to meet total demand.

When the demand exceeds the output capacity of all online generators, economic dispatch schemes will use generators with quick ramp-up times to ensure stability of the power grid. These generators are very inefficient. One goal of direct load control (DLC) as an economic dispatch scheme is to reduce the deviation of actual demand from the forecasted demand.

Consider the direct load control model:
\begin{equation}
x_{k+1} =  x_k + u_k + \mu_k + d_k
\end{equation}
Here, $x_k \in \Rbb$ represents the power consumption of a unit at time $k$, where a unit can be a household, an HVAC system for a building, or a sector of the power grid. 
$u_k \in \R$ represents the direct load control signal at time $k$. $\mu_k \in \R$ represents the affine term which generates our nominal demands at time $k$; if $u_k \equiv 0$ and $d_k \equiv 0$, then $\mu_k$ creates our forecasted demand. Finally, $d_k$ represents the disturbance at time $k$. In this model, disturbances from the nominal demand persist, and DLC policies must be employed to return the power consumption to the nominal demand.

Now, consider different sampling rates. That is, we suppose our controller is only able to receive measurements every $N$ time steps. However, it is still able to issue control commands at every time step. We wish to design a controller that makes use of the available measurements to optimally issue control commands to a sector of the power grid.

The subsampled system can be modeled in a Markov jump linear system (MJLS) framework. To define optimality, we consider the $\Hc_\infty$ norm of MJLSs, as defined in \cite{Seiler2005,Ishii2008}. In our application, the $\Hc_\infty$ norm represents a worst case estimate of how much the true power consumption will deviate from the power consumption used for unit commitment. This worst case estimate is a function of the uncertainty in the load forecast.

Recent results in the analysis of MJLS gives us the optimal $\Hc_\infty$ controller for subsampled scalar systems~\cite{Dong2014}. Thus, we can analyze how the performance of direct load control is affected by different sampling schemes. For example, Figure~\ref{fig:equi_fig_1} gives us the performance for equidistant sampling. 
Thus, the formulation allows us to quantify the value of a higher sampling rate to the utility company.
\begin{figure}[ht]
	\begin{center}
	\includegraphics[scale=0.4]{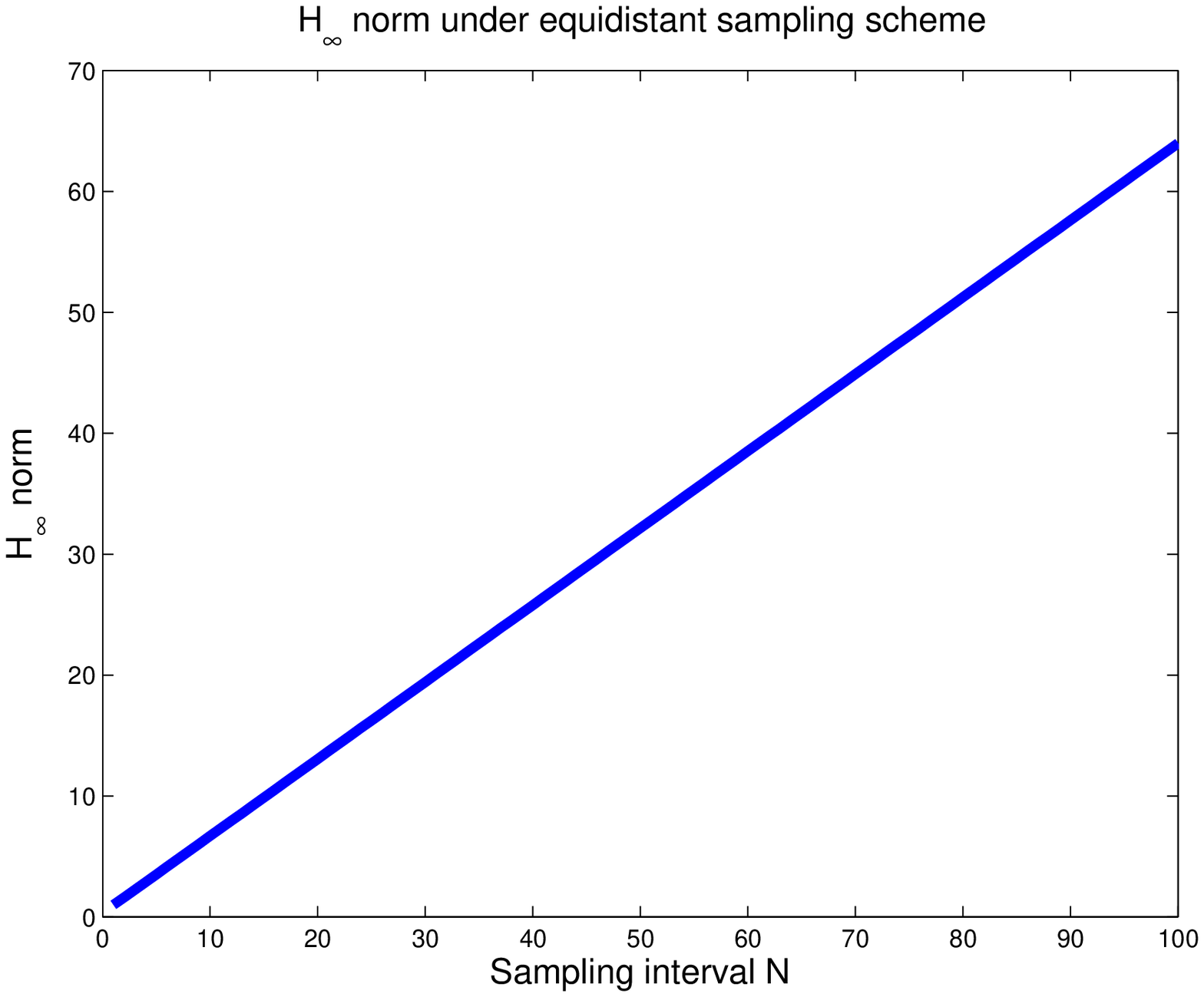}
	\end{center}
        \caption{ The $\Hc_\infty$ norm for the equidistant sampling scheme as a function of the sampling interval $N$. A higher sampling interval $N$ corresponds to a lower sampling frequency, i.e. the utility company receives less data.}
	\label{fig:equi_fig_1}
\end{figure}